\documentclass[12pt]{amsart}
\usepackage{color}
\usepackage{verbatim,amsmath,amssymb}

\usepackage{eucal}

\begin{document}
\newtheorem{cor}{Corollary}[section]
\newtheorem{theorem}[cor]{Theorem}
\newtheorem{prop}[cor]{Proposition}
\newtheorem{lemma}[cor]{Lemma}
\newtheorem*{lemma*}{Lemma}
\theoremstyle{definition}
\newtheorem{defi}[cor]{Definition}
\theoremstyle{remark}
\newtheorem{remark}[cor]{Remark}
\newtheorem{example}[cor]{Example}

\def\W{\mathcal{W}}
\def\A{\mathrm{A}}
\def\T{\mathrm{T}}
\def\Sym{\mathrm{Sym}}
\def\Id{\mathrm{Id}}

\def\f{\varphi}
\def\e{\varepsilon}
\def\d{\mathrm{d}}
\def\s{\sigma}
\def\de{\delta}
\def\O{\Omega}
\def\o{\omega}
\def\R{\mathrm{R}}
\def\a{\alpha}
\def\b{\beta}
\def\g{\gamma}
\def\l{\lambda}
\def\la{\langle}
\def\ra{\rangle}
\def\res{|}
\def\SO{{\rm SO}}
\def\U{{\rm U}}
\def\Hol{{\rm Hol}}
\def\Ric{{\rm Ric}}
\def\End{{\rm End}}
\def\tr{{\rm tr}}
\def\id{{\rm Id}}
\def\S{{\rm S}}
\def\vol{{\rm vol}}
\newcommand{\be}{\begin{equation}}
\newcommand{\ee}{\end{equation}}
\def\beq{\begin{eqnarray*}}
\def\eeq{\end{eqnarray*}}
\def\p{\psi}
\def\.{{\cdot}}
\def\n{\nabla}
\def\ci{C^\infty}
\def\RR{\mathcal{R}}
\def\B{\mathrm{Ric}_0}
\def\nt{\tilde\nabla}
\def\ck{\mathcal{CK}}
\def\RM{\mathbb{R}}
\def\C{\mathbb{C}}

\title{Conformal Killing 2-forms on 4-dimensional manifolds}
\author{Adri\'an Andrada}
\address{Adri\'an Andrada, FAMAF-CIEM, Universidad Nacional de C\'ordoba, Ciudad Universitaria, 5000 C\'ordoba, Argentina}
\email{andrada@famaf.unc.edu.ar}
\author{Mar\'ia Laura Barberis}
\address{Mar\'ia Laura Barberis, FAMAF-CIEM, Universidad Nacional de C\'ordoba, Ciudad Universitaria, 5000 C\'ordoba, Argentina}
\email{barberis@famaf.unc.edu.ar}
\author{Andrei Moroianu}\thanks{The first and second authors  were partially supported by
CONICET, ANPCyT and SECyT-UNC (Argentina), and the third author was partially supported by the ANR-10-BLAN 0105 grant of the
Agence Nationale de la Recherche (France).}
\address{Andrei Moroianu \\ Laboratoire de Math\'ematiques de Versailles, UVSQ, CNRS, Universit\'e Paris-Saclay, 78035 Versailles, France }
\email{andrei.moroianu@math.cnrs.fr}
\date{\today}
\begin{abstract}
We study 4-dimensional simply connected Lie groups $G$ with left-invariant Riemannian metric $g$ admitting non-trivial conformal Killing 2-forms. We show that either the real line defined by such a form is invariant under the group action, or the metric is half conformally flat. In the first case, the problem reduces to the study of invariant conformally K\"ahler structures, whereas in the second case, the Lie algebra of $G$ belongs (up to homothety) to a finite list of families of metric Lie algebras.
\end{abstract}

\maketitle

\section{Introduction}

{\em Conformal Killing forms} (sometimes called twistor forms) on Riemannian manifolds generalize to higher degrees the notion of conformal vector fields. They can be characterized by the fact that their covariant derivative with respect to the Levi-Civita connection is completely determined by  their exterior derivative and co-differential (cf. Definition \ref{dcf} below). For basic facts on conformal Killing forms, see \cite{uwe} and references therein. 

A conformal Killing form whose co-differential vanishes is called {\em Killing form}. The metric structure of 4-dimensional Riemannian manifolds with Killing 2-forms was recently described in \cite{gm}, while the case of left invariant Killing 2-forms on Lie groups with left invariant metrics was studied in \cite{BDS}.  The $3$-dimensional Lie 
groups with left invariant metrics carrying left invariant conformal Killing  $2$-forms were classified in \cite{ABD}.  

In this paper we will be mainly concerned with the problem of conformal Killing 2-forms on 4-dimensional Riemannian manifolds $(M,g)$. After obtaining some general results in Section \ref{3}, we consider the particular case where $M=G$ is a simply connected Lie group and $g$ is a left invariant Riemannian metric on $G$ induced by a scalar product $\la.,.\ra$ on the Lie algebra $\mathfrak{g}$ of $G$. 
Our main result is Theorem \ref{main}, where we show that if a simply connected $4$-dimensional Lie group with left invariant metric carries a non-trivial conformal Killing 2-form $\o$, then it either has a left invariant conformally K\"ahler structure, or the metric Lie algebra $(\mathfrak{g},\la.,.\ra)$ belongs to an explicit list of half-conformally flat metric Lie algebras.
Note that we do not assume any invariance property for $\o$. As a matter of fact, every {\em left invariant} conformal Killing 2-form on a 4-dimensional Lie group is automatically parallel (cf. Lemma \ref{const}).

The proof of Theorem \ref{main} goes roughly as follows. One can first assume that the conformal Killing form $\o$ is self-dual. If the line generated by $\o$ is $G$-invariant, then $\o$ divided by its norm defines a left invariant conformally K\"ahler structure. If $\mathbb{R}\o$ is not $G$-invariant, then the space of self-dual conformal Killing forms is at least 2-dimensional, and in this case the manifold is half-conformally flat. We then use the classification of these manifolds by S. Maier \cite{maier} and V. De Smedt, S. Salamon \cite{salamon} and show that, conversely, every half-conformally flat simply connected $4$-dimensional Lie group carries non-parallel conformal Killing forms. 

\section{Preliminaries}
\subsection{Four-dimensional Euclidean geometry}

Let $(E,\la .,.\ra)$ be an oriented 4-dimen\-sional Euclidean vector space. We identify $E$ and $E^*$ by means of the Euclidean metric. Similarly, $\Lambda^2E$ will be identified with the space of skew-symmetric endomorphisms:
$$X\wedge Y=(Z\mapsto \la X,Z\ra Y-\la Y,Z\ra X),\qquad\forall X,Y,Z\in E.$$ 
We endow the exterior algebra of $E$ with the unique scalar product equal to $\la ., . \ra$ on $E$ with respect to which the interior and exterior product with vectors are adjoint endomorphisms:
$$\la X\lrcorner \alpha,\beta\ra=\la\alpha,X\wedge\beta\ra,\qquad\forall \alpha,\beta\in\Lambda^* E,\ \forall X\in E.$$

The volume form of $E$ will be denoted by $\vol\in\Lambda^4E$. 
The Hodge duality is the automorphism $*$ of $\Lambda\* E$ defined by 
$$\a\wedge*\b=\la\a,\b\ra\, \vol ,\qquad\forall \alpha,\beta\in\Lambda^* E.$$
It satisfies $*^2=(-1)^k$ on $\Lambda^k E$ and its eigenspaces for the eigenvalues $\pm1$ are denoted by $\Lambda^2_\pm E$.
It is easy to check that as endomorphisms, every element of $\Lambda^2_+E$ commutes with every element of $\Lambda^2_-E$, and that 
\be\label{com}[\Lambda^2_+ E, \Lambda^2_+ E]\subset \Lambda^2_+ E,\qquad [\Lambda^2_- E, \Lambda^2_- E]\subset \Lambda^2_- E.\ee

For $X\in E$ the following formulas hold:
\be\label{dual} X\lrcorner*\a=(-1)^k*(X\wedge \a),\qquad X\wedge*\a=(-1)^{k-1}*(X\lrcorner \a),\qquad\forall \a\in\Lambda^kE.
\ee
If $e_i$ denotes an orthonormal basis of $E$ then 
\be\label{sum}e_i\wedge(e_i\lrcorner\a)=k\a,\qquad e_i\lrcorner(e_i\wedge \a)=(4-k)\a,\qquad\forall \a\in\Lambda^kE
\ee
(here and in the sequel we use Einstein's summation convention over repeating subscripts).

For every $\a\in\Lambda^2E$ and $X,Y\in E$ we have the following commutator relation
\be\label{aco}[\a,X\wedge Y]=\a(X)\wedge Y+X\wedge\a(Y).\ee
From \eqref{com} we thus get:
\be\label{pm}\a(X)\wedge Y+X\wedge\a(Y)\in\Lambda^2_\pm E,\qquad\forall \a\in\Lambda^2_\pm E.\ee

\subsection{The curvature tensor of 4-dimensional Riemannian manifolds}

Let $(M,g)$ be an oriented, connected  4-dimensional Riemannian manifold, with Levi-Civita covariant derivative $\n$ and curvature tensor
$$\R_{X,Y}Z=[\n_X,\n_Y]Z-\n_{[X,Y]}Z,\qquad\forall X,Y,Z\in \ci(\T M).$$
We will sometimes see the curvature tensor as a fully covariant tensor by the formula 
$$\R(X,Y,Z,V):=g(\R_{X,Y}Z,V).$$ 
As such, the curvature tensor is symmetric by pairs and satisfies the Bianchi identities.
The Ricci tensor is defined by 
$$\Ric(X,Y):=\tr\left\{V\mapsto \R_{V,X}Y\right\}=g(\R_{e_i,X}Y,e_i)$$
with respect to some  local orthonormal basis $e_i$ of $\T M$.
We also introduce the scalar curvature $\S:=\tr_g\Ric=\Ric(e_i,e_i)$, the trace-free Ricci tensor $\B:=\Ric-\frac14 \S$ and the (symmetric) curvature operator 
$\RR:\Lambda^2 M\to \Lambda^2 M$ defined by 
$$g(\RR(X\wedge Y),Z\wedge V):=g(\R_{Y,X}Z,V).$$ 
Every symmetric tensor $A\in\Sym^2(M)$ defines a symmetric endomorphism $\tilde A$ of $\Lambda^2 M$ by 
$$\tilde A(X\wedge Y):= A(X)\wedge Y+X\wedge A(Y)=A\circ (X\wedge Y)+(X\wedge Y)\circ A.$$
If $A$ is trace-free then $\tilde A$ maps $\Lambda^2_\pm M$ to $\Lambda^2_\mp M$. The curvature operator decomposes as follows under the natural $\SO(4)$ action:
\be\label{r1}\RR=\frac \S{12}\Id+\frac12\widetilde \B+\W^++\W^-,\ee
where $\W^\pm:\Lambda^2_\pm M\to \Lambda^2_\pm M$ are the self-dual and anti-self-dual Weyl tensors. Summarizing, we have the following matrix decomposition of $\RR$:
\be\label{r}\RR=\begin{pmatrix} \frac \S{12}\Id+\W^+& \frac12 \widetilde \B\\ & \\ \frac12 \widetilde \B& \frac \S{12}\Id+\W^-\end{pmatrix}.\ee

Let now $\theta$ be a vector field on $M$ and denote by $Q:=(\n\theta)^s$ the symmetric part of the endomorphism $\n\theta$. For later use, we now derive a formula relating the covariant derivatives of $\d\theta$ and $(\n\theta)^s$.

\begin{lemma} \label{l} The following formula holds:
$$\nabla_X\d\theta=2(\d^\n Q)(X)+2\R_{X,\theta},$$
where $(\d^\n Q)(X):=e_i\wedge(\n_{e_i}Q)(X)$.
\end{lemma}
\begin{proof} Taking a further covariant derivative in the equation $\n_X\theta=Q(X)+\tfrac12\d\theta(X)$ and skew-symmetrizing yields
$$\R_{X,Y}\theta=(\n_XQ)(Y)-(\n_YQ)(X)+\tfrac12((\n_X\d\theta)(Y)-(\n_Y\d\theta)(X)).$$
We take the exterior product with $Y$ in this relation and sum over a local orthonormal basis $Y=e_i$ to obtain:
\be\label{rx}e_i\wedge\R_{X,e_i}\theta=e_i\wedge(\n_XQ)(e_i)-(\d^\n Q)(X)+\tfrac12e_i\wedge(\n_X\d\theta)(e_i)-\tfrac12e_i\wedge(\n_{e_i}\d\theta)(X).\ee
We now investigate each term in this relation. From the first Bianchi identity we get 
$$\R(X,\theta,e_i,e_j)=\R(X,e_i,\theta,e_j)-\R(X,e_j,\theta,e_i),$$
whence by skew-symmetry and \eqref{sum}:
\beq e_i\wedge\R_{X,e_i}\theta&=&e_i\wedge e_j \R(X,e_i,\theta,e_j)=\tfrac 12e_i\wedge e_j (\R(X,e_i,\theta,e_j)-\R(X,e_j,\theta,e_i))\\
&=&\tfrac 12e_i\wedge e_j\R(X,\theta,e_i,e_j)=\tfrac 12e_i\wedge e_j\R_{X,\theta}e_i=\R_{X,\theta}.\eeq
The first term in the right hand side of \eqref{rx} vanishes since $\n_XQ$ is symmetric. By \eqref{sum} again, the third term is equal to $\n_X\d\theta$, and the fourth term equals to
$$-\tfrac12e_i\wedge(\n_{e_i}\d\theta)(X)=\tfrac12X\lrcorner (e_i\wedge(\n_{e_i}\d\theta))-\tfrac12\n_X\d\theta=-\tfrac12\n_X\d\theta.$$
Substituting these results in \eqref{rx} yields 
$$\R_{X,\theta}=-(\d^\n Q)(X)+\tfrac12\n_X\d\theta,$$
which is equivalent to the desired formula.
\end{proof}

\section{Conformal Killing forms in dimension 4}\label{3}

\begin{defi}\label{dcf}[cf. \cite{uwe}] A conformal Killing $p$-form on a $n$-dimensional Riemannian manifold $(M,g)$ is a $p$-form $\o\in \Omega^p(M)$ satisfying the equation 
\be\label{ckf} \n_X\o=X\wedge\a+ X\lrcorner \b,\qquad\forall X\in\T M,
\ee
for some forms $\a\in\Omega^{p-1}M$ and $\b\in\Omega^{p+1}M$.
\end{defi}
By taking the exterior or interior product with $X$ and summing over an orthonormal basis, we immediately get $\a=-\frac1{n+1-p}\delta\o$ and $\b=\frac1{p+1}\d \o$.

The following results are folklore, see e.g. \cite{uwe} and references therein. 

\begin{lemma}\label{conf-inv} If $\o$ is a conformal Killing $p$-form with respect to a metric $g$, then $\tilde\o:=f^{p+1}\o$ is a conformal Killing $p$-form with respect to the conformal metric $\tilde g:=f^2 g$. 
\end{lemma}

\begin{theorem}\label{parallel}
There exists a connection $\nabla^K$ on $\Lambda^pM\oplus\Lambda^{p-1}M\oplus\Lambda^{p+1}M\oplus\Lambda^pM$ (called the Killing connection), such that for every conformal Killing $p$-form $\o$, the tuple $(\o,\delta\o, \d\o, \Delta\o)$ is $\nabla^K$-parallel. Conversely, the first component of a $\nabla^K$-parallel section is a conformal Killing form. Moreover, the curvature of the Killing connection vanishes if the Riemannian curvature of $(M,g)$ vanishes. 
\end{theorem}

\begin{cor}\label{vanish}
A conformal Killing form which vanishes on some non-empty open set, vanishes identically.
\end{cor}

\begin{cor}\label{max} The space $\mathcal{CK}_p$ of conformal Killing $p$-forms is a finite dimensional vector space of dimension $\leq \binom{n+2}{p+1}$. Moreover, the equality holds if $M$ is simply connected and $g$ is conformally flat.
\end{cor}
\begin{proof}
The inequality follows directly from Theorem \ref{parallel} and the fact that for every $x\in M$ the map $\mathcal{CK}_p\to (\Lambda^pM\oplus\Lambda^{p-1}M\oplus\Lambda^{p+1}M\oplus\Lambda^pM)_x$ given by $\o\mapsto (\o,\delta\o, \d\o, \Delta\o)_x$ is injective. 

If $M$ is simply connected and $g$ is conformally flat, a classical result of Kuiper \cite{kuiper} shows that the conformal class $[g]$ contains a flat metric $\tilde g$. By Theorem \ref{parallel} the Killing connection of $(M,\tilde g)$ is flat, so $\mathcal{CK}_p(M,\tilde g)$ has dimension $\binom{n+2}{p+1}$. The last statement thus follows from Lemma \ref{conf-inv}.
\end{proof}

Assume now that $n=4$ and $p=2$. 

\begin{lemma}\label{asd}
If $\o$ is a conformal Killing $2$-form on $M$, its self-dual and anti-self-dual parts $\o_\pm$ are also conformal Killing 2-forms. 
\end{lemma}
\begin{proof}
Taking the Hodge dual in \eqref{ckf} and using \eqref{dual} yields 
\be\label{ckf1} \n_X(*\o)=-X\lrcorner(*\a)+X\wedge *\b,\qquad\forall X\in\T M,
\ee
so 
$$\n_X(\o_\pm)=\tfrac12X\wedge(\a\pm*\b)+ \tfrac12X\lrcorner (\b\mp*\a),\qquad\forall X\in\T M.$$

\end{proof}

Let $\o$ be a non-trivial conformal Killing $2$-form on $M$. 
After changing the orientation if necessary, we can thus assume that $\o\in\O^2_+(M)$ is self-dual and non-trivial. The conformal Killing equation reduces to 
\be\label{ck}\nabla_X\o=(X\wedge\theta)_+,\qquad\forall X\in \T M,
\ee
for some $1$-form $\theta:=\a+*\b\in \O^1(M)$. By taking a contraction in \eqref{ck} and using \eqref{dual} and \eqref{sum} we readily obtain
\be\label{do}\d \o=e_i\wedge(e_i\wedge\theta)_+=\tfrac12 e_i\wedge *(e_i\wedge\theta)=-\tfrac12*( e_i\lrcorner(e_i\wedge\theta))=-\tfrac32*\theta,\ee
and consequently 
\be\label{deo}\delta\o=-*\d *\o=-*\d\o=-\frac32\theta.\ee

\begin{lemma}\label{const}
If $\o$ has constant norm then it is parallel on $M$. 
\end{lemma}
\begin{proof}
We can assume that the norm of $\omega$ is non-zero, so $\omega$ is non-degenerate (being self-dual). Taking the scalar product with $\o$ in \eqref{ck} yields for every $X\in \T M$:
$$0=g(\o,\n_X\o)=g(\o,(X\wedge\theta)_+)=g(\o,X\wedge\theta)=\o(X,\theta).$$
As $\o$ is non-degenerate, this gives $\theta=0$.
\end{proof}

The following result is an observation of Pontecorvo \cite{pontecorvo}:

\begin{lemma}\label{ponte} On the open set $M_0$ where $\omega$ is non-vanishing, the conformal metric $\tilde g:=|\o|^{-2}g$ is K\"ahler.
\end{lemma}
\begin{proof}
Taking $p=2$ and $f=|\o|^{-1}$, in Lemma \ref{conf-inv}, we see that the form $\tilde \o$ is conformal Killing and has constant norm (equal to 1) with respect to $\tilde g$. By Lemma \ref{const}, it is parallel with respect to the Levi-Civita covariant derivative $\tilde \n$ of $\tilde g$. Correspondingly, the endomorphism $J$ of $\T M$ defined by 
$$\tilde g(JX,Y)=\sqrt 2 \tilde\omega(X,Y)$$
is a $\tilde g$-orthogonal almost complex structure which is $\tilde \n$-parallel.
\end{proof}

Note that $M_0$ is a dense open set, by Corollary \ref{vanish}.

\begin{cor}\label{coro} If $\o$ is a non-trivial self-dual conformal Killing form, then $\W^+(\o)$ is proportional to $\o$ at any point, so there exists a smooth function $\l$ defined on $M_0$ such that $\W^+(\o)=\l\o$ on $M_0$. Moreover, the two other eigenvalues of $\W^+$ at any point $x$ of $M_0$ are equal to $-\frac12\l(x)$.
\end{cor}
\begin{proof} It is well known that $\W^+$ is conformally covariant as endomorphism of $\Lambda^2_+(M)$, in the sense that if $\tilde g=f^2 g$ then $\tilde\W^+=f^{-2}\W^+$. As the statement is conformally invariant, Lemma \ref{ponte} shows that we can assume that $\o$ is parallel, and equal to the K\"ahler form of a parallel Hermitian structure $J$ compatible with the orientation. Then $\Lambda^2_+(M)$ decomposes in an orthogonal direct sum
$$\Lambda^2_+(M)=\mathbb{R}\o\oplus[[\Lambda^{(2,0)}(M)]],$$
where $[[\Lambda^{(2,0)}(M)]]$ consist of $2$-forms anti-commuting with $J$. The curvature operator of any K\"ahler manifold vanishes on $[[\Lambda^{(2,0)}(M)]]$ so by \eqref{r} we must have 
\be\label{wk}\W^+\res_{[[\Lambda^{(2,0)}(M)]]}=-\tfrac{\S}{12}\Id.\ee
As $\tr(\W^+)=0$, the third eigenvalue of $\W^+$ is $\frac{\S}{6}$. At each point of $M_0$ we either have $\S=0$, and then $\W^+$ is identically zero at that point, or $\S\ne 0$ and then the eigenspace of $\W^+$ corresponding to $\frac{\S}{6}$ is the orthogonal of $[[\Lambda^{(2,0)}(M)]]$ in $\Lambda^2_+(M)$, i.e. $\mathbb{R}\o$. In both cases we get the desired conclusion.
\end{proof}

\begin{lemma}\label{co} If $\o$ and $f\o$ are non-trivial conformal Killing forms on $M$ then $f$ is constant.
\end{lemma}
\begin{proof} As before, we can assume that $\o$ is self-dual and non-trivial. We then have 
\be\label{ck1}\nabla_X\o=(X\wedge\theta)_+,\qquad \nabla_X(f\o)=(X\wedge\tilde\theta)_+,\qquad\forall X\in \T M,
\ee
for some 1-forms $\theta$ and $\tilde\theta$. By subtracting these two equations we get 
$$X(f)\o=(X\wedge(\theta-f\tilde\theta))_+,\qquad\forall X\in \T M.$$
If $(\theta-f\tilde\theta)$ were non-zero at some point $x$, the right hand side takes any value in $\Lambda^2_+(\T_x M)$ as $X$ runs over $\T_xM$, which is of course impossible. Thus $\theta-f\tilde\theta=0$ and consequently $f$ is constant.
\end{proof}

This observation has an important consequence:

\begin{prop}\label{sd} If the vector space of self-dual conformal Killing forms on $M$ has dimension greater than $1$, then $M$ is self-dual, i.e. $\W^+\equiv 0$.
\end{prop}
\begin{proof}
Let $\o_1$ and $\o_2$ be linearly independent self-dual conformal Killing forms on $M$. The sets of points $M_i$ where $\o_i$ is non-vanishing are dense on in $M$ by \cite{uwe}, and so is their intersection. We claim that the set $M_3$ of points $x$ where $\o_1(x)$ and $\o_2(x)$ are linearly independent is dense in $M_1\cap M_2$, and thus in $M$. Indeed, if $\o_1(x)$ and $\o_2(x)$ were collinear for every $x$ in some open subset $U\subset M_1\cap M_2$, then by Lemma \ref{co} we would have $\o_1=\l\o_2$ on $U$ for some constant $\l$, so by Corollary \ref{vanish}, the conformal Killing form  $\o_1-\l\o_2$ would vanish identically on $M$.

Now, by Corollary \ref{coro}, the self-dual Weyl tensor $\W^+$ has to vanish at every point of $M_3$ and thus $\W^+\equiv 0$ on $M$.
\end{proof}

It turns out that in the particular case of 2-forms on 4-manifolds, the Killing connection mentioned in Theorem \ref{parallel} splits into a direct sum of connections, $\nabla^{K_\pm}$, defined on $\Lambda^2_\pm M\oplus\T M\oplus\Lambda^2_\mp M$. We describe $\nabla^{K_+}$ below (omitting the subscript "+" in order to keep notations simpler).

Taking first the covariant derivative in \eqref{ck} and skew-symmetrizing yields
$$\R_{X,Y}\o=(Y\wedge\n_X\theta-X\wedge\n_Y\theta)_+,\qquad \forall X,Y\in \T M.$$
We then make the interior product with $Y$ in this relation and sum over a local orthonormal basis $Y=e_i$. Using \eqref{dual} and \eqref{sum} we obtain:
\beq e_i\lrcorner\R_{X,e_i}\o&=&\tfrac 12e_i\lrcorner(e_i\wedge\n_X\theta-X\wedge\n_{e_i}\theta)+\tfrac12e_i\lrcorner*(e_i\wedge\n_X\theta-X\wedge\n_{e_i}\theta)\\
&=&\tfrac32\n_X\theta-\tfrac12\n_X\theta-X\delta\theta-\tfrac12*(e_i\wedge X\wedge\n_{e_i}\theta)\\
&=&\n_X\theta+\tfrac12 X\lrcorner*\d\theta.
\eeq
On the other hand, we can write $\n\theta=(\n\theta)^s+\tfrac12\d\theta$, where $(\n\theta)^s$ denotes the symmetric part of the endomorphism $\n\theta$, so the previous relation reads
\be\label{nt}e_i\lrcorner\R_{X,e_i}\o=(\n\theta)^s(X)+(\d\theta)_+(X).
\ee
Now, using the first Bianchi identity we get for every tangent vector $X$:
$$\RR(\o)(X)=-\tfrac12\R_{e_i,\o(e_i)}X=\tfrac12\R_{\o(e_i),X}e_i+\tfrac12\R_{X,e_i}\o(e_i)=\R_{X,e_i}\o(e_i),$$
and thus
\be\label{rr} e_i\lrcorner\R_{X,e_i}\o=\R_{X,e_i}\o(e_i)-\o(\R_{X,e_i}e_i)=(\RR(\o)-\o\circ\Ric)(X).\ee
Taking \eqref{r1} into account we compute:
\beq (\RR(\o)-\o\circ\Ric)&=&\tfrac{\S}{12}\o+\tfrac12(\B\circ\o+\o\circ\B)+\W^+(\o)-\o\circ(\B+\tfrac\S4\Id)\\
&=&-\tfrac{\S}{6}\o+\W^+(\o)+\tfrac12[\B,\o],
\eeq
where $[\B,\o]:=\B\circ\o-\o\circ\B$ is a symmetric endomorphism. From this equation, together with \eqref{nt} and \eqref{rr} we thus get 
\be\label{sys}\begin{cases}(\n\theta)^s=\tfrac12[\B,\o]\\
(\d\theta)_+=-\tfrac{\S}{6}\o+\W^+(\o).
\end{cases}
\ee
This can be rewritten as
\be\label{nt1}\n\theta=\tfrac12[\B,\o]+\tfrac12(-\tfrac{\S}{6}\o+\W^+(\o))+\tfrac12\s,
\ee
where $\s:=(\d\theta)_-\in \Lambda^2_-M$ denotes the anti-self-dual part of $\d\theta$. 

We claim that the covariant derivative of $\s$ is a linear expression of the previous data $\o$ and $\theta$, involving the curvature of $M$ and its first derivative. From Lemma \ref{l} we get 
$$\n_X\d\theta=(\d^\n[\B,\o])(X)+2\R_{X,\theta},$$
so using the fact that $\d\theta=(\d\theta)_++(\d\theta)_-=(-\tfrac{\S}{6}\o+\W^+(\o))+\s$, we obtain from \eqref{ck}
\beq\n_X\s&=&(\d^\n[\B,\o])(X)+2\R_{X,\theta}-\n_X(-\tfrac{\S}{6}\o+\W^+(\o))\\
&=&e_i\wedge[\n_{e_i}\B,\o](X)+e_i\wedge[\B,(e_i\wedge\theta)_+](X)+2\R_{X,\theta}\\
&&-(-\tfrac{X(\S)}{6}\o+(\n_X\W^+)(\o))-(-\tfrac{\S}{6}(X\wedge\theta)_++\W^+((X\wedge\theta)_+)),
\eeq
thus showing our claim. 

In the particular case where $(M,g)$ is self-dual and Einstein (i.e. $\W^+=0$, $\B=0$, $\S=$ constant), this equation simplifies to 
\be\label{sympl}\n_X\s=2\R_{X,\theta}+\tfrac{\S}{6}(X\wedge\theta)_+.
\ee
On the other hand, \eqref{r1} yields
$$\R_{X,\theta}=-\RR(X\wedge\theta)=-\tfrac1{12}S(X\wedge\theta)-\W^-((X\wedge\theta)_-),$$ 
whence 
\be\label{sympl1}\n_X\s=-(\tfrac{\S}6+2\W^-)((X\wedge\theta)_-).
\ee
Summarizing, we have shown that every conformal Killing 2-form $\o$ on a self-dual Einstein 4-manifold satisfies the system
\be\label{sys2}\begin{cases}\n_X\o=(X\wedge\theta)_+\\
\n_X\theta=-\tfrac{\S}{12}\o(X)+\tfrac1{2}\s(X)\\
\n_X\s=-(\tfrac{\S}6+2\W^-)((X\wedge\theta)_-).
\end{cases}
\ee
Inspired by the previous computations we can now state the following:

\begin{theorem} \label{tsd}
Let $(M^4,g)$ be an oriented simply connected Riemannian manifold which is self-dual ($\W^+=0$) and K\"ahler-Einstein with respect to a complex structure $J$ compatible with the opposite orientation.
Then, if $M$ is not flat, the space of self-dual conformal Killing forms on $M$ is $8$-dimensional.
\end{theorem}
\begin{proof} Let $\O:=g(J\cdot,\cdot)$ denote the K\"ahler form of $J$.
Consider the following connection on the rank $8$ vector bundle $\Lambda^2_+M\oplus\T M\oplus\mathbb{R}$:
$$\nabla^K_X\begin{pmatrix}\o\\\theta\\f\end{pmatrix}:=\begin{pmatrix}\n_X\o-(X\wedge\theta)_+\\\n_X\theta+\tfrac{\S}{12}\o(X)-\tfrac f{2}J(X)\\X(f)+\tfrac S4\O(X,\theta)\end{pmatrix}.$$
We compute the second covariant derivative (at some point where the vector field $Y$ is $\n$-parallel):
\begin{equation*}\begin{split}(\nabla^K)^2_{X,Y}\begin{pmatrix}\o\\\theta\\f\end{pmatrix}&=\nabla^K_X\begin{pmatrix}\n_Y\o-(Y\wedge\theta)_+\\
\n_Y\theta+\tfrac{\S}{12}\o(Y)-\tfrac f{2}J(Y)\\
Y(f)+\tfrac S4\O(Y,\theta)\end{pmatrix}\\
&=\begin{pmatrix}\n^2_{X,Y}\o-(Y\wedge\n_X\theta)_+-(X\wedge(\n_Y\theta+\tfrac{\S}{12}\o(Y)-\tfrac f{2}J(Y)))_+\\
\{\n^2_{X,Y}\theta+\tfrac{\S}{12}(\n_X\o)(Y)-\tfrac1{2}X(f)J(Y)+\tfrac S{12}(\n_Y\o-(Y\wedge\theta)_+)(X)\\
\hskip 5cm-\tfrac12(Y(f)+\tfrac S4\O(Y,\theta))J(X)\}\\
X(Y(f))+\tfrac S4\O(Y,\n_X\theta)+\tfrac S4\O(X,\n_Y\theta+\tfrac{\S}{12}\o(Y)-\tfrac f{2}J(Y))\end{pmatrix}
\end{split}\end{equation*}
and observing that the last row is symmetric in $X$ and $Y$ (as the endomorphisms associated to $\O\in\Omega^2_-(M)$ and $\o\in\Omega^2_+(M)$ commute), we obtain:
\be\label{rrr}\R^{\nabla^K}_{X,Y}\begin{pmatrix}\o\\\theta\\f\end{pmatrix}=\begin{pmatrix}\R_{X,Y}\o+\tfrac f2(X\wedge J(Y)-Y\wedge J(X))_+-\tfrac{\S}{12}(X\wedge\o(Y)-Y\wedge \o(X))_+\\
\R_{X,Y}\theta-\tfrac S{12}(Y\wedge\theta)_+(X)+\tfrac S{12}(X\wedge\theta)_+(Y)
-\tfrac S8 (J(Y)\wedge J(X))(\theta)\\
0\end{pmatrix}.\ee
By \eqref{pm} the first row of \eqref{rrr} equals 
$$\R_{X,Y}\o-\tfrac{\S}{12}(X\wedge\o(Y)-Y\wedge \o(X))=\R_{X,Y}\o-\tfrac{\S}{12}[\o,X\wedge Y].$$
Since $\o$ is self-dual, it commutes with $\W^-(X\wedge Y)$ for every $X,Y\in\T M$. From\eqref{aco} and  \eqref{r1} we get
$$\R_{X,Y}\o=-[\RR(X\wedge Y),\o]=-([\tfrac{S}{12}X\wedge Y+\W^-(X\wedge Y),\o]=[\o,\tfrac{S}{12}X\wedge Y],$$
thus showing that the  first row of \eqref{rrr} vanishes.

In order to compute the second row, we notice that by \eqref{wk} the symmetric operator $\tfrac S{12}+\W^-$ of $\Lambda^2_- M$ vanishes on the orthogonal complement  of $\O$ in $\Lambda^2_-M$, so it is proportional to the orthogonal projector on $\O$. Moreover, since $M$ is K\"ahler-Einstein, we have the well known relation $\RR(\Omega)=\rho$ (the Ricci form), so by \eqref{r} we get that the proportionality factor is $\tfrac S4$. Consequently:
\be\label{wm}\W^-(X\wedge Y)=-\tfrac S{12}(X\wedge Y)_-+\tfrac S{8}\O(X,Y)\O.\ee
We therefore obtain from \eqref{r1}:
\beq\R_{X,Y}\theta&=&-\RR(X\wedge Y)(\theta)=(-\tfrac S{12}X\wedge Y-\W^-(X\wedge Y))(\theta)\\
&=&(-\tfrac S{12}(X\wedge Y)_+-\tfrac S{8}\O(X,Y)\O)(\theta),\eeq
so the second row of \eqref{rrr} equals
$$-\tfrac S{12}((X\wedge Y)_+(\theta)+(Y\wedge\theta)_+(X)+(X\wedge\theta)_+(Y))-\tfrac S{8}(\O(X,Y)\O
+J(Y)\wedge J(X))(\theta).
$$
In order to prove that this expression vanishes, it is enough to show that for every $\theta,X,Y\in\T M$ the following relations hold:
\be\label{s1} (X\wedge Y)_+(\theta)+(Y\wedge\theta)_+(X)+(X\wedge\theta)_+(Y)=\tfrac32\theta\lrcorner*(X\wedge Y)\ee
and 
\be\label{s2}*(X\wedge Y)-J(X)\wedge J(Y)+\O(X,Y)\O=0.\ee
From \eqref{dual} we get
$$ (X\wedge Y)_+(\theta)=\tfrac12(g(X,\theta)Y-g(Y,\theta)X)+\tfrac12*(\theta\wedge X\wedge Y).$$
The cyclic sums in $\theta,X,Y$ of the first two terms cancel with each other, and the cyclic sum of the third term is 3 times itself. This proves \eqref{s1}. In order to prove \eqref{s2}, it is sufficient, by bi-linearity, to check it for $X$, $Y$ elements of some adapted basis, which is straightforward.

We thus have checked that the connection $\nabla^K$ is flat, which by simple connectedness implies that it admits an 8-dimensional space of parallel sections. Clearly the first component of any $\nabla^K$-parallel section is a self-dual conformal Killing form. 

We will now show that conversely, if $\o$ is a self-dual conformal Killing form on $M$, then it is equal to the first component of some $\nabla^K$-parallel section. Let $\theta$ and $\sigma$ be the vector field and anti-self-dual 2-form defined by \eqref{ck} and \eqref{nt1}. From \eqref{sys2} and \eqref{wm} we get
$$\n_X\s=-\tfrac S4\O(X,\theta)\O,\qquad\forall X\in\T M.$$
By taking a further covariant derivative we obtain:
$$\R_{X,Y}\sigma=-\tfrac S4(\O(Y,\n_X\theta)-\O(X,\n_Y\theta))\O=0$$
since the second row of \eqref{sys2} shows that the endomorphism $\n\theta$ is skew-symmetric and commutes with $J$. This relation shows that the anti-self-dual 2-form $\s$ commutes with the image through $\RR$ of every $2$-form, in particular with $\RR(\O)=\tfrac S4\O$. Since $M$ is assumed to be non-flat we must have $S\ne 0$ (otherwise $\RR=0$ by \eqref{r1} and \eqref{wm}). Thus $\s$ commutes with $\O$, so $\s=f\O$ for some function $f$. The system \eqref{sys2} together with \eqref{wm} then show that the section $(\o,\theta,f)$ is $\nabla^K$-parallel.
\end{proof}

\section{(Half-)conformally flat metrics on four dimensional Lie groups} 

Conformally flat left invariant metrics on 4-dimensional Lie groups were classified by S. Maier in \cite{maier}.
Half-conformally flat, non-conformally flat, left invariant metrics on 4-dimensional Lie groups were classified by V. De Smedt and S. Salamon in \cite{salamon}. We briefly recall their classifications below.

\subsection{Conformally flat Lie groups of dimension $4$}\label{4.1}
According to \cite{maier}, simply connected $4$-dimensional conformally flat Lie groups fall into four different families that correspond to types II, III, IV and VI from \cite{maier}.

\medskip

$\bullet$ Type II: The groups $\RM \ltimes \operatorname{SU}(2)$. These groups are isomorphic to $\RM \times \operatorname{SU}(2)$ and the corresponding conformally flat metric is the Riemannian product of the standard metrics on $\RM$ and $\mathbb{S}^3$.

\medskip

$\bullet$ Type III: The groups  $\RM \ltimes_{\pi} \RM ^3$, where $\pi : \RM \to \operatorname{Aut} (\RM ^3)$ is given by $\pi(t) x= e^t \theta (t) x$ for $t\in \RM, x\in \RM^3$, and $\theta : \RM \to \operatorname{SO} (3)$ is a Lie group homomorphism. The Lie algebras of these Lie groups are semidirect products $\RM \ltimes_{\phi} \RM ^3$ where $\phi: \RM \to \operatorname{End} (\RM^3)$ is given by
\[    \phi(t)= t I +A(t), \qquad A(t) \in \mathfrak{so} (3).                           \]
Setting $A:= A(1) \in \mathfrak{so}(3), \, B:=\phi(1)=I+A$  and $e_0$ a unit vector orthogonal to $\RM ^3$, the above Lie algebra can be written as $\RM e_0 \ltimes_B \RM^3$, with Lie bracket $[e_0,x]=Bx, \, x\in \RM^3$. There exists an orthonormal basis $\{e_1, e_2, e_3\}$ of $\RM^3$ such that $B$ takes the following form for some $\alpha\geq 0$:
\[   B=\begin{pmatrix} 1& -\alpha & 0 \\
\alpha & 1& 0\\
0 &0& 1
\end{pmatrix} .               \]
We will denote by $\mathfrak{g}_{\alpha}$ the corresponding Lie algebra, with Lie brackets:
\begin{align*}
[e_0,e_1]&=e_1+\alpha e_2, \\
[e_0,e_2]&=-\alpha e_1+ e_2, \\
[e_0 ,e_3]&=e_3.
\end{align*}
According to the notation in \cite{abdo}, $\mathfrak g_{\alpha}$ is isomorphic to $\mathfrak r_{4,1,1}$ for $\alpha=0$ and to $\mathfrak r'_{4,1/\alpha , 1/\alpha}$ for $\alpha >0$. Therefore, these are non-isomorphic Lie algebras.

\medskip

$\bullet$ Type IV: The groups  $\RM^2 \ltimes_{\pi} \RM ^2$, where $\pi : \RM^2 \to \operatorname{Aut} (\RM ^2)$ is given by $\pi(t_1,t_2) x= e^{t_1}\, \theta (t_1,t_2) x$ for $(t_1,t_2)\in \RM^2, x\in \RM^2$, and $\theta : \RM^2 \to \operatorname{SO} (2)$ is a Lie group homomorphism. The Lie algebras of these Lie groups are semidirect products $\RM^2 \ltimes_{\phi} \RM ^2$ where $\phi: \RM^2 \to \operatorname{End} (\RM^2)$ is given by
\[    \phi(t_1,t_2)= t_1 I +A(t_1,t_2), \qquad A(t_1,t_2) \in \mathfrak{so} (2).                           \]
It follows that 
\[ A(t_1,t_2)= \begin{pmatrix} 0 & -\alpha(t_1,t_2) \\ \alpha(t_1,t_2)& 0 \end{pmatrix} , \quad \text{for } \alpha(t_1,t_2) \in \RM, \] with respect to an orthonormal basis $\{ f_1, f_2\}$ of the abelian ideal $\RM^2$. The Lie algebra structure is determined by the action of $e_1:=(1,0)$ and $e_2:=(0,1)$ on span$\{ f_1, f_2\}$.  If we denote by  $\mathfrak g_{a,b}$ the Lie algebra corresponding to  $a= \alpha(e_1), \; 
b= \alpha(e_2)$, the Lie bracket on $\mathfrak g_{a,b}$  is given by 
\begin{align*} [e_1, f_1]&= f_1+af_2, \qquad \quad  [e_2, f_1]= b f_2, \\
[e_1, f_2]&= -af_1+f_2,\quad \quad \, [e_2, f_2]= -b f_1.
\end{align*}
Using the notation from \cite{abdo}, it follows that $\mathfrak g_{a,0}\cong \RM \times \mathfrak r'_{3, 1/a}$ for $a\neq 0$ and $\mathfrak g_{a,b}\cong \mathfrak{aff}(\C)$ for $b\neq 0$. The Lie algebras $\mathfrak g_{a,0}$ are pairwise non-isomorphic for $a>0$ and the metric Lie algebra $\mathfrak g_{a,b}$ is isometric to $\mathfrak g_{a,-b}$ for $b\neq 0$. 

\medskip

$\bullet$ Type VI:  Flat Lie groups, whose structure has been determined by Milnor \cite{Mi}. In dimension $4$ there is only one such simply connected Lie group, with Lie algebra generated by an orthonormal basis $\{ e_0, e_1, e_2, e_3\}$ and the following Lie brackets:
\[ [e_1,e_2]=e_3, \qquad\qquad [e_1,e_3]=-e_2.  \]
This Lie algebra is isomorphic to $\RM \times \mathfrak e (2)$, where $\mathfrak e (2)$ is the Euclidean Lie algebra.

\subsection{Half-conformally flat non-conformally flat Lie groups of dimension $4$}\label{4.2}
According to \cite{salamon}, the Lie algebra and the metric (up to a constant multiple) of such a group are necessarily of the type $\mathfrak{g}(a,b)$ for $a=1$ or $a=\frac12$, where in an orthonormal basis $\{e_1, \dots , e_4\}$ the Lie algebra structure of $\mathfrak{g}(a,b)$ is given by:
 \begin{align*}  [e_1,e_2]&=ae_2-be_3, & [e_1,e_3]&=be_2+ae_3,\\
[e_1,e_4]&=2ae_4, &    [e_2,e_3]&=-e_4,\\
[e_2,e_4]&=0, &    [e_3,e_4]&=0,
\end{align*}
Note that in \cite{salamon} a different set of parameters, $(\lambda,k)$, is used, which correspond to the parameters $(a,b)$ above by defining $a:=\dfrac1k, \; b:=\dfrac{\lambda}k$.


The exterior derivative is given  as follows, where we identify vectors with one-forms via the fixed inner product and $e_{ij}$ denotes $e_i\wedge e_j$:
\begin{align*}
\d e_1&=0,& \d e_2&=-ae_{12}-be_{13},\\
\d e_3&=be_{12}-ae_{13}, & \d e_4&=-2ae_{14}+e_{23}.
\end{align*}

Recall  that a locally conformally K\"ahler (lcK) structure on a manifold $M$ is defined by an integrable complex structure $J$ and a compatible 2-form $\Omega$ (in the sense that $\Omega(\cdot,J\cdot)$ is a Riemannian metric) such that $\d\Omega=\theta\wedge\Omega$ for some closed $1$-form $\theta$, called the {\em Lee form}. 

\begin{lemma}\label{lckab}
The orthogonal complex structures $J_+$ and $J_-$ defined on $\mathfrak{g}(a,b)$ by $J_\e(e_1)=e_4$ and $J_\e(e_2)=\e e_3$ induce left invariant conformally K\"ahler structures on the simply connected group $G(a,b)$.  Moreover, the structure $J_\e$ is K\"ahler for $a=\frac \e2$.
\end{lemma}
\begin{proof}
We have for every $\e\in \{-1,1\}$:
$$[e_1+iJ_\e e_1,e_2+iJ_\e e_2]=[e_1+ ie_4,e_2+i\e e_3]=(a+i\e b)(e_2+i\e e_3)=(a+i\e b)(e_2+iJ_\e e_2),$$
so $J_\e$ induce integrable complex structures on $G(a,b)$. Moreover, if $\Omega_\e:=\la J_\e\cdot,\cdot\ra$ denote the corresponding fundamental 2-forms, then from the above formulas we get:
$$\d \Omega_\e=\d(e_{14}+\e e_{23})=-(1+2\e a)e_{123}=(\e-2a)e_1\wedge\Omega_\e.$$
This shows that the structure $(J_\e,\O_\e)$ is conformally K\"ahler, with Lee form 
$$\theta_\e:=(\e-2a)e_1.$$
The last statement is a direct consequence of this formula.
\end{proof}

\section{Conformal Killing forms on 4-dimensional Lie groups}

Assume now that $G$ is a 4-dimensional simply connected Lie group and $g$ is some left-invariant Riemannian metric on $G$. We fix some orientation on $G$ and define the space $\ck$ of conformal Killing forms on $G$. From Lemma \ref{asd} we know that $\ck=\ck_+\oplus\ck_-$, where $\ck_\pm:=\ck\cap \Omega^2_\pm G$. 

\begin{lemma}\label{llck} Every Lie group $G$ with a left invariant conformally K\"ahler metric $g$ admits self-dual conformal Killing $2$-forms. 
\end{lemma}
\begin{proof}
Indeed, if $\tilde g = e^{-2f}g$ is K\"ahler (with $f$ a smooth function on $G$) and $\tilde \Omega$ is the corresponding K\"ahler form, then $\omega=e^{-3f}\tilde \Omega$ is a conformal Killing $2$-form, according to Lemma \ref{conf-inv}.
\end{proof}

\begin{theorem}\label{main} Assume that $\ck\ne 0$. Then, up to an orientation change, one of the following exclusive possibilities occurs:
\begin{enumerate}
\item $(G,g)$ is conformally flat and its metric Lie algebra belongs to the list described in Section \ref{4.1}. For each solution, the space $\ck$ has the maximal possible dimension $20$. 
\item $(G,g)$ is half-conformally flat and belongs to one of the two families $G(1,b)$ or $G(\frac12,b)$ described in Section \ref{4.2}. The manifold $G(\frac12,b)$ is isometric to the complex hyperbolic plane for every $b$, and has $\dim(\ck_+)=8$, $\dim(\ck_-)=1$. The manifold $G(1,b)$ has $\dim(\ck_+) \ge 1$, $\dim(\ck_-)=1$.
\item $(G,g)$ is not half-conformally flat, but has an invariant conformally K\"ahler structure $(\O,\theta)$ satisfying $\d\O=\theta\wedge\O$. The space
$\ck_+$ is the line generated by $e^{f}\O$, where $f\in \mathcal{C}^\infty(G)$ is some primitive of $\frac\theta 2$. If $G$ has another invariant conformally K\"ahler structure compatible with the opposite orientation (e.g. when $(G,g)=G(a,b)$ with $a\ne \frac12$, $a\ne 1$, cf. Lemma \ref{lckab}), then $\dim(\ck_-)=1$. If not, then $\dim(\ck_-)=0$.
\end{enumerate}
\end{theorem}
\begin{proof} We may assume that $\ck_+\ne 0$, otherwise we just change the orientation of $G$. If $\dim(\ck_+)\ge 2$, Proposition \ref{tsd} implies that $G$ is half-conformally flat, so we are in one of the first two cases. The assertions concerning the spaces of conformal Killing forms follow from Corollary \ref{max} in case (1) and from Theorem \ref{sd} and Lemma \ref{lckab} in case (2). The fact that $G(\frac12,b)$ is isometric to the complex hyperbolic plane was noticed in \cite{salamon}. 

It remains to treat the case when $\ck_+$ is 1-dimensional, generated by some self-dual conformal Killing 2-form $\o$, and $(G,g)$ is not self-dual. Denote by $F:=|\o|^2$ the square norm of $\o$. Since $\ck_+$ is preserved by left translations with elements in $G$, it follows that for every $\gamma\in G$, there exists a non-zero constant $c_\gamma$ such that \be\label{hg}L_\gamma^*(\o)=c_\gamma\o.\ee 
The map $\gamma\mapsto c_\gamma$ is clearly a group morphism from $G$ to $\mathbb{R}^*_+$. This shows, in particular, that $\o$ does not vanish on $G$. Moreover, since the metric is left invariant, the above relation gives 
\be\label{hg2}L_\gamma^*F=|L_\gamma^*\omega|^2=(c_\gamma)^2F,\qquad \forall \gamma\in G.\ee 

By Lemma \ref{ponte} we see that the metric $\tilde g:=\frac 1F g$ is K\"ahler, with K\"ahler form $\tilde\O:=\sqrt 2 F^{-\frac32}\o$. Therefore 
$(g,\O,\theta)$ is an lcK structure for $\O:=\frac{\sqrt2}{\sqrt F}\o$ and $\theta:=\frac{\d F}{F}$. From \eqref{hg} and \eqref{hg2} we see that $\O$ and $\theta$ are is left invariant. We finally see that $\o=e^f\O$ for $f:=\frac12(\ln F-\ln 2)$, which satisfies $\d f=\frac\theta2$. The last statement follows from Lemma \ref{llck}.
\end{proof}

The complete classification of $4$-dimensional metric Lie algebras whose associated simply connected Lie group is conformally K\"ahler is not yet available in general. Some recent results on left invariant lcK structures on Lie groups can be found in  \cite{achk,ao,ka,sawai}.


\begin{thebibliography}{9}

\bibitem{achk} D. V. Alekseevsky, V. Cortes, K. Hasegawa, Y. Kamishima, {\sl Homogeneous locally conformally K\"ahler and Sasaki manifolds},  Internat. J. Math. {\bf 26} (6) (2015), 1--29. 

\bibitem{abdo} A. Andrada, M. L. Barberis, I. G. Dotti,  G. Ovando, {\sl Product structures on four dimensional solvable Lie algebras}, Homology Homotopy Appl. {\bf 7} (1) (2005), 9--37.

\bibitem{ABD}  A. Andrada, M. L. Barberis,  I. Dotti, {\sl Invariant solutions to the conformal Killing-Yano equation on Lie groups}, J. Geom.  Phys. {\bf 94} (2015), 199--208.

\bibitem{ao} A. Andrada, M. Origlia, {\sl Locally conformally K\"ahler structures on unimodular Lie groups},  Geom. Dedicata {\bf 179} (1) (2015), 197--216.

\bibitem{BDS} M. L. Barberis, I. Dotti, O. Santill\'an, {\sl The Killing-Yano equation on Lie groups}, 
Class. Quantum Grav. {\bf 29} (2012), 1--10.

\bibitem{salamon} V. De Smedt, S. Salamon, {\sl Anti-self-dual metrics on Lie groups}, Differential geometry and integrable systems (Tokyo, 2000), 63--75, Contemp. Math. {\bf 308}, Amer. Math. Soc., Providence, RI, 2002.

\bibitem{gm} P. Gauduchon, A. Moroianu, {\sl Killing 2-forms in dimension 4}, arXiv:1506.04292.

\bibitem{ka} H. Kasuya, {\sl Vaisman metrics on solvmanifolds and Oeljeklaus-Toma manifolds}, Bull. London Math. Soc. {\bf 45} (1) (2013), 15--26.

\bibitem{kuiper} N. Kuiper, {\sl On conformally flat spaces in the large}, Ann. of Math. {\bf 50} (1950), 916--924.

\bibitem{maier} S. Maier, {\sl Conformally flat Lie groups}, Math. Z. {\bf 228} (1998), 155--175.

\bibitem{Mi} J. Milnor, {\sl Curvature of left invariant metrics on Lie groups}, Adv. Math. \textbf{21} (1976), 293--329.

	
\bibitem{pontecorvo} M. Pontecorvo, {\sl On twistor spaces of anti-self-dual Hermitian surfaces,} Trans. Amer. Math. Soc. {\bf 331} (1992), 653--661.

\bibitem{sawai} H. Sawai, {\sl Locally conformal K\"ahler structures on compact solvmanifolds}, Osaka J. Math. {\bf 49} (2012), 1087-–1102.

\bibitem{uwe} {U. Semmelmann,} {\sl Conformal Killing forms on Riemannian manifolds, }  Math. Z. {\bf 243} (2003), 503--527.




\end{thebibliography}
\end{document}